\documentclass[reqno,tbtags,a4paper,12pt]{amsart}

\usepackage[bbgreekl]{mathbbol}
\usepackage{amsfonts}

\DeclareSymbolFontAlphabet{\mathbb}{AMSb}
\DeclareSymbolFontAlphabet{\mathbbl}{bbold}

\usepackage{amsmath,amssymb,amsthm,amsfonts,amscd,array,graphicx,color}
\usepackage[in]{fullpage}

\newcommand{\D}{\Delta}
\newcommand{\hD}{\widetilde{\D}}
\newcommand{\hphi}{\widetilde{\varphi}}

\newcommand{\R}{\mathbb{R}}
\newcommand{\K}{\mathbb{K}}
\newcommand{\E}{\mathbb{E}}
\newcommand{\bL}{\mathbbl{\Lambda}}
\newcommand{\C}{\mathbb{C}}
\newcommand{\Z}{\mathbb{Z}}
\newcommand{\Q}{\mathbb{Q}}
\newcommand{\X}{\mathbb{X}}

\newcommand{\fE}{\mathbf{E}}
\newcommand{\fV}{\mathbf{V}}
\newcommand{\fF}{\mathbf{F}}

\newcommand{\bS}{\mathbb{S}}

\newcommand{\CV}{\mathcal{V}}
\newcommand{\CP}{\mathcal{P}}
\newcommand{\CO}{\mathcal{O}}

\newcommand{\bx}{\mathbf{x}}

\newcommand{\bn}{\mathbf{n}}
\newcommand{\bm}{\mathbf{m}}
\newcommand{\bell}{\boldsymbol{\ell}}

\newcommand{\Ln}{\mathop{\mathrm{Ln}}\nolimits}

\newcommand{\Isom}{\mathop{\mathrm{Isom}}\nolimits}

\newcommand{\vol}{\mathop{\mathrm{vol}}\nolimits}
\newcommand{\dist}{\mathop{\mathrm{dist}}\nolimits}
\newcommand{\reg}{\mathrm{reg}}
\newcommand{\sing}{\mathrm{sing}}

\renewcommand{\Im}{\mathop{\mathrm{Im}}\nolimits}

\newtheorem*{sbc}{Strong Bellows Conjecture}
\newtheorem{theorem}{Theorem}[section]
\newtheorem{propos}[theorem]{Proposition}
\newtheorem{cor}[theorem]{Corollary}
\newtheorem{lem}[theorem]{Lemma}

\newtheorem*{bell-conj}{Bellows conjecture}
\theoremstyle{definition}

\newtheorem{remark}[theorem]{Remark}
\numberwithin{equation}{section}

\author{Alexander A. Gaifullin, Leonid Ignashchenko}

\thanks{The work of the first listed author is supported by the Russian Foundation for Basic Research (grant 16-51-55017) and by a grant of the President of the Russian Federation (grant MD-2907.2017.1).}

\title{Dehn invariant of flexible polyhedra}

\date{}

\address{Alexander A. Gaifullin:}

\address{Steklov Mathematical Institute of Russian Academy of Sciences, Gubkina str. 8, Moscow, 119991, Russia}

\address{Moscow State University, Leninskie gory 1, Moscow, 119991, Russia}

\address{Skolkovo Institute of Science and Technology, Skolkovo Innovation Center, build.~3, Moscow,  143026, Russia}

\address{Institute for Information Transmission Problems (Kharkevich Institute), Bolshoy Karetny per. 19, build.~1, Moscow, 127051, Russia}

\email{agaif@mi.ras.ru}

\address{Leonid Ignashchenko:}

\address{Moscow State University, Leninskie gory 1, Moscow, 119991, Russia}

\email{lenya.ignashenko@yandex.ru}

\begin{document}


\begin{abstract}
We prove that the Dehn invariant of any flexible polyhedron in $n$-dimen\-sio\-nal Euclidean space, where $n\ge 3$, is constant during the flexion. For $n=3$ and~$4$ this implies that any flexible polyhedron remains scissors congruent to itself during the flexion. This proves the Strong Bellows Conjecture posed by Connelly in 1979. It was believed that this conjecture was disproved by Alexandrov and Connelly in 2009. However, we find an error in their counterexample. Further, we show that the Dehn invariant of a flexible polyhedron in either sphere~$\bS^n$ or Lobachevsky space~$\bL^n$, where $n\ge 3$, is constant during the flexion if and only if this polyhedron satisfies the usual Bellows  Conjecture, i.\,e., its volume is constant during every flexion of it. Using previous results due to the first listed author, we deduce that the Dehn invariant is constant during the flexion for every bounded flexible polyhedron in odd-dimensional Lobachevsky space and for every  flexible polyhedron with sufficiently small edge lengths in any space of constant curvature of dimension greater than or equal to~$3$.
\end{abstract}

\maketitle

\section{Introduction}

Let $\X^n$ be one of the three $n$-dimensional spaces of constant curvature, that is, Euclidean space~$\E^n$ or sphere~$\bS^n$ or Lobachevsky space~$\bL^n$.

A \textit{flexible polyhedron\/} is an $(n-1)$-dimensional closed polyhedral surface~$P$ in~$\X^n$ that admits a \textit{flexion}, i.\,e., a non-trivial continuous deformation~$P_t$ not induced by an isometry of the whole space~$\X^n$ such that every face of~$P_t$ remains congruent to itself during the deformation. Notice that the surface~$P$ is not required to be embedded, though embedded flexible polyhedra, which are called \textit{flexors}, are of a special interest. In the spherical case $\X^n=\bS^n$ it is usually reasonable to consider only flexible polyhedra contained in the open hemisphere~$\bS^n_+$. 

First examples of flexible polyhedra are Bricard's flexible self-intersecting octahedra, see~\cite{Bri97}. By now, examples of flexible polyhedra are known in all spaces of constant curvature of all dimensions, see~\cite{Gai13}. The first example of an embedded flexible polyhedron in~$\E^3$ was constructed by Connelly~\cite{Con77}. His construction can be easily extended to~$\bS^3$ and~$\bL^3$, cf.~\cite{Kui78}. The first listed author constructed examples of embedded flexible polyhedra in hemispheres~$\bS^n_+$ of all dimensions, see~\cite{Gai15a}. However, it is still unknown if embedded flexible polyhedra exist in Euclidean spaces and Lobachevsky spaces of dimensions greater than~$3$.

The \textit{Bellows Conjecture} (see~\cite{Con78}, \cite{Kui78}) claims that the volume of any flexible polyhedron is constant during the flexion. Here and further we always assume that the dimension $n$ is greater than or equal to~$3$. Under the volume of a polyhedron~$P$ we always mean the $n$-dimensional volume enclosed by the $(n-1)$-dimensional polyhedral surface~$P$. If $P$ is not embedded, then a natural concept of a \textit{generalized oriented volume} should be applied, see Section~\ref{subsection_vol} for details.
By now, the Bellows Conjecture is proved in the following cases:
\begin{itemize}
\item flexible polyhedra in Euclidean spaces~$\E^n$ (Sabitov for $n=3$, see~\cite{Sab96}, \cite{Sab98a}, \cite{Sab98b}, and see~\cite{CSW97} for another proof, and the first listed author for $n\ge 4$, see~\cite{Gai11}, \cite{Gai12}),
\item bounded flexible polyhedra in odd-dimensional Lobachevsky spaces~$\bL^{2n+1}$, see~\cite{Gai15b},
\item flexible polyhedra with sufficiently small edge lengths in any of the spaces~$\bL^n$ and~$\bS^n$, see~\cite{Gai17}.
\end{itemize}
Counterexamples to the Bellows Conjecture in open hemispheres~$\bS^n_+$ were constructed by Alexandrov~\cite{Ale97} for $n=3$ and by the first listed author~\cite{Gai15a} for  $n\ge 4$. It is still unknown whether the Bellows Conjecture is true in even-dimensional Lobachevsky spaces.

Recall that two $n$-dimensional polytopes~$P$ and~$Q$ in~$\X^n$ are said to be \textit{scissors congruent} if either of them can be divided into convex polytopes with pairwise disjoint interiors, 
$$P=P_1\cup\cdots\cup P_k,\qquad Q=Q_1\cup\cdots\cup Q_k,$$ such that $P_j$ is congruent to~$Q_j$ for every~$j$. The scissors congruence relation can be easily extended to non-embedded polyhedra, see Section~\ref{subsection_scissors}. Dehn~\cite{Deh02} showed that polytopes of equal volume are not necessarily scissors congruent, thus answering Hilbert's Third Problem. Namely,  he proved that to be scissors congruent two polytopes must necessarily have equal \textit{Dehn invariants}, where the Dehn invariant of a polytope~$P$ is, by definition, the following element of the group $\R\otimes_{\Z}(\R/\pi\Z)=\R\otimes_{\Q}(\R/\pi\Q)$:
\begin{equation}\label{eq_Dehn_initial}
\D(P)=\sum_{\dim F=n-2}\vol_{n-2}(F)\otimes\varphi_F\,.
\end{equation}
Here the sum is taken over all $(n-2)$-dimensional faces~$F$ of~$P$, and~$\varphi_F$ is the dihedral angle of~$P$ at~$F$.

In~\cite{Con79} Connelly proposed the following conjecture, which is naturally called the \textit{Strong Bellows Conjecture}:
\begin{sbc}
Any flexible polyhedron remains scissors congruent to itself during the flexion.
\end{sbc}
A weak form of the Strong Bellows Conjecture says that \textit{the Dehn invariant of any flexible polyhedron remains constant during the flexion.}
Sydler~\cite{Syd65}  proved that two polytopes in~$\E^3$ are scissors congruent  if and only if they have equal volumes and Dehn invariants. Jessen~\cite{Jes68} proved that the same holds for polytopes in~$\E^4$. Hence in~$\E^3$ and in~$\E^4$ the weaker form of the Strong Bellows Conjecture is equivalent to the original form of it. The question of whether any two polyhedra of equal volumes and Dehn invariants are scissors congruent to each other is still open in Euclidean spaces~$\E^n$, where $n\ge 5$, and in non-Euclidean spaces~$\bS^n$ and~$\bL^n$, where $n\ge 3$, see~\cite{Dup01}.

Alexandrov~\cite{Ale10} checked that the Dehn invariants of Bricard's flexible octahedra are constant during the flexion.

Alexandrov and Connelly~\cite{AlCo11} stated that they constructed a counterexample to the Strong Bellows Conjecture, that is, a flexible polyhedron in~$\E^3$ with non-constant Dehn invariant. Nevertheless, their proof of the fact that the Dehn invariant of the constructed polyhedron is non-constant contained a fatal error. We discuss this example and point out the error in Section~\ref{section_example}.

The main result of the present paper is the proof of the weak form of the Strong Bellows Conjecture.

\begin{theorem}\label{thm_main}
The Dehn invariant of any flexible polyhedron in~$\E^n$, where $n\ge 3$, remains constant during the flexion.
\end{theorem}

\begin{cor}
Suppose that $n$ is either $3$ or~$4$. Then any flexible polyhedron  in~$\E^n$ remains scissors congruent to itself during the flexion.
\end{cor}

For non-Euclidean spaces, we can prove that the Strong Bellows Conjecture is true whenever the Bellows Conjecture is true.

\begin{theorem}\label{thm_conditional}
Let $P$ be a polyhedron in either $\bS^n$ or~$\bL^n$, where $n\ge 3$. Assume that the Bellows Conjecture is true for~$P$, that is, the generalized oriented volume is constant during every flexion~$P_t$ of~$P$. Then the Dehn invariant is also constant during every flexion~$P_t$ of~$P$. 
\end{theorem}

\begin{remark}
In this theorem, it is important that the Bellows Conjecture is true for \textit{all} flexions of~$P$. If we know only that there exists a flexion~$P_t$ of~$P$ such that the volume of~$P_t$ is constant, then we cannot deduce that the Dehn invariant of~$P_t$ is constant.
\end{remark}

As we have already mentioned, the Bellows Conjecture is true for all bounded polyhedra in odd-dimensional Lobachevsky spaces, and for all polyhedra with sufficiently small edge lengths in any non-Euclidean space, see~\cite{Gai15b},~\cite{Gai17}. The following corollaries are immediate consequences of Theorem~\ref{thm_conditional} above, Theorem~1.1 in~\cite{Gai15b}, and Theorem~1.1 in~\cite{Gai17}.

\begin{cor}
The Dehn invariant of any bounded flexible polyhedron in~$\bL^n$, where $n$ is odd and $n\ge 3$, remains constant during the flexion.
\end{cor}

\begin{cor}
Let $\X^n$ be either~$\bS^n$ or~$\bL^n$, where $n\ge 3$. Let $P_t$ be a simplicial flexible polyhedron in~$\X^n$ with $m$ vertices such that all edges of~$P_t$ have lengths smaller than $2^{-m^2(n+4)}$. Then the Dehn invariant of~$P_t$ is constant during the flexion.
\end{cor}

\section{Basic definitions and notation}

\subsection{Flexible polyhedra}
We always normalize metrics of sphere~$\bS^n$ and of Lobachevsky space~$\bL^n$ so that they have sectional curvatures~$1$ and~$-1$, respectively.
To deal with~$\bS^n$ and~$\bL^n$ simultaneously, we shall conveniently introduce the parameter~$\varepsilon$ that is equal to~$1$ for~$\bS^n$ and to~$-1$ for~$\bL^n$.
We identify~$\bS^n$ and~$\bL^n$ with their standard vector models, i.\,e., with the unit sphere in Euclidean vector space~$\R^{n+1}$, and with the half of the hyperboloid $\langle x,x\rangle=1$, $x_0>0$ in  pseudo-Euclidean vector space~$\R^{1,n}$.
We denote by $x_0,\ldots,x_n$ the standard coordinates in  either~$\R^{n+1}$ or~$\R^{1,n}$
such that the (pseudo-)Euclidean scalar product in these coordinates writes as
$$
\langle x,y\rangle=x_0y_0+\varepsilon(x_1y_1+\cdots+x_ny_n).
$$

Let $\Delta^k$ be an affine simplex with vertices $v_0,\ldots,v_k$. A mapping $g\colon\Delta^k\to\X^n$, where $\X^n$ is either~$\bS^n$ or~$\bL^n$, will be called \textit{pseudo-linear} if it satisfies
$$
g(\lambda_0v_0+\cdots+\lambda_kv_k)=\frac{\lambda_0g(v_0)+\cdots+\lambda_kg(v_k)}{|\lambda_0g(v_0)+\cdots+\lambda_kg(v_k)|}
$$
for any $\lambda_0,\ldots,\lambda_k$ such that $\lambda_0+\cdots+\lambda_k=1$, where $|x|=\sqrt{\langle x,x\rangle}$.

If $\X^n=\E^n$, then we denote by $x_1,\ldots,x_n$ the standard  Cartesian coordinates in~$\E^n$.

A $k$-dimensional \textit{pseudo-manifold\/} is a finite simplicial complex~$K$ such that
\begin{enumerate}
\item every simplex of~$K$ is contained in a $k$-dimensional simplex,
\item every $(k-1)$-dimensional  simplex of~$K$ is contained in exactly two $k$-dimensional simplices,
\item $K$ is \textit{strongly connected}, i.\,e., the complement~$K\setminus\mathrm{Sk}^{k-2}(K)$ is connected, where $\mathrm{Sk}^{k-2}(K)$ is the $(k-2)$-skeleton of~$K$.
\end{enumerate}
A pseudo-manifold~$K$ is said to be \textit{oriented\/} if its $k$-dimensional simplices are endowed with compatible orientations.

Let $K$ be an $(n-1)$-dimensional pseudo-manifold. A \textit{polyhedron\/} of combinatorial type~$K$ is a mapping $P\colon K\to\X^n$ whose restriction to every face of~$K$ is affine linear in the case~$\X^n=\E^n$ and pseudo-linear in the cases $\X^n=\bS^n$ and~$\X^n=\bL^n$.

A \textit{flexion} of a polyhedron of combinatorial type~$K$ is a continuous family of polyhedra $P_t\colon K \to\X^n$, where $t$ runs over some interval $(a,b)$, such that for any $t_1,t_2\in(a,b)$, every face of~$P_{t_1}$ is congruent to the corresponding face of~$P_{t_2}$. The flexion is said to be \textit{non-trivial} if $P_{t_1}$  cannot be obtained from~$P_{t_2}$ by an isometry of~$\X^n$ for any sufficiently close to each other~$t_1$ and~$t_2$. A polyhedron is called \textit{flexible} if it admits a non-trivial flexion.

The above definitions allow only simplicial polyhedra. This is not a restriction, since any polyhedral surface has a simplicial subdivision. If the initial polyhedral surface is flexible, then its simplicial subdivision will still admit all the same flexions and, possibly, some new flexions will appear. Hence Theorems~\ref{thm_main} and~\ref{thm_conditional} for arbitrary flexible polyhedra will follow from these theorems for simplicial flexible polyhedra. So further we always deal with simplicial polyhedra only.

We say that a polyhedron $P\colon K\to\X^n$ has \textit{non-degenerate faces} if for any simplex $\sigma=[v_0\ldots v_k]$ of~$K$ the image $P(\sigma)$ is a non-degenerate simplex of the same dimension~$k$. For $\E^n$, this condition means that the points $P(v_0),\ldots,P(v_k)$ are affinely independent. For $\bS^n$ and~$\bL^n$, this condition means that the vectors $P(v_0),\ldots,P(v_k)$ are linearly independent (in the vector model). Further, we shall always consider only polyhedra with non-degenerate faces. This is also not a restriction, since there is a standard procedure that turns an arbitrary flexible polyhedron to a flexible polyhedron with non-degenerate faces and with the same underlying geometric polyhedral surface, see~\cite[Section~11]{Gai15b}.

\begin{remark}
There exists a more general  definition of a polyhedron called \textit{cycle polyhedron}. In this definition a pseudo-manifold~$K$ is replaced with an arbitrary simplicial cycle, see~\cite{Gai12},~\cite{Gai17}. All results of the present paper can be easily extended to the case of cycle polyhedra. 
\end{remark}

\subsection{Configuration spaces}

Let $m$ and~$r$ be the numbers of vertices and edges of~$K$, respectively, and  let $\fV(K)$ and $\fE(K)$ be the sets of vertices and edges of~$K$, respectively. A polyhedron $ P\colon K\to \X^n$ is determined solely by the positions of its vertices. 
For each vertex $v$ of~$K$, we put $\bx_v=P(v)$, and we denote by $x_{v,1},\ldots,x_{v,n}$ (in the case $\X^n=\E^n$) and by $x_{v,0},\ldots,x_{v,n}$ (in the cases $\X^n=\bS^n$ and $\X^n=\bL^n$) the coordinates of~$\bx_v$. We put $N=mn$ if $\X^n=\E^n$ and $N=m(n+1)$ if $\X^n=\bS^n$ or $\X^n=\bL^n$, and consider the space~$\R^N$ with the coordinates $x_{v,j}$, where $v\in\fV(K)$, $j=(0,)1,\ldots,n$. We consider $x_{v,j}$ as coordinates on the space of all polyhedra $P\colon K\to \X^n$, and identify every polyhedron~$P$ with the corresponding point in~$\R^N$. Below, $\R^N$ is referred to as the \textit{space of all polyhedra of combinatorial type~$K$}. 

Now, fix a set $\bell=(\ell_e)$ of positive real numbers indexed by edges $e\in\fE(K)$. 
Let $\Sigma=\Sigma_{K,\bell,\X^n}\subseteq\R^N$ be the subset consisting of all polyhedra  $P\colon K\to\X^n$ with the given set of edge lengths~$\bell$. (This subset may be empty.)
The set~$\Sigma$ will be called the \textit{configuration space} of all polyhedra of the given combinatorial type and edge lengths.

If $\X^n=\E^n$, then $\Sigma$ is the algebraic subset of~$\R^N=\R^{mn}$ given by the $r$ equations
\begin{equation}
|\bx_u-\bx_v|^2=\sum_{j=1}^n(x_{u,j}-x_{v,j})^2=\ell_{[uv]}^2,\qquad [uv]\in\fE(K).
\end{equation}

If $\X^n=\bS^n$, then $\Sigma$ is the algebraic subset of~$\R^N=\R^{m(n+1)}$ given by the $m+r$ equations
\begin{align}
\langle\bx_u,\bx_u\rangle&=\sum_{j=0}^nx_{u,j}^2=1,&u&\in\fV(K),\\
\langle\bx_u,\bx_v\rangle&=\sum_{j=0}^nx_{u,j}x_{v,j}=\cos\ell_{[uv]},& [uv]&\in\fE(K).
\end{align}

If $\X^n=\bL^n$, then $\Sigma$ is the semi-algebraic subset of~$\R^N=\R^{m(n+1)}$ given by the $m+r$ equations
\begin{align}\label{eq_system_L1}
\langle\bx_u,\bx_u\rangle&=x_{u,0}^2-\sum_{j=1}^nx_{u,j}^2=1,&u&\in\fV(K),\\
\label{eq_system_L2}\langle\bx_u,\bx_v\rangle&=x_{u,0}x_{v,0}-\sum_{j=1}^nx_{u,j}x_{v,j}=\cosh\ell_{[uv]},& [uv]&\in\fE(K),
\end{align}
and one inequality
\begin{equation}\label{eq_system_L3}
x_{u_0,0}>0,
\end{equation}
where $u_0$ is any chosen vertex of~$K$. It is easy to see that, since $K$ is connected, equations~\eqref{eq_system_L2} and inequality~\eqref{eq_system_L3} imply the inequalities $x_{u,0}>0$ for all other vertices $u\in\fV(K)$. 

Though in the case~$\X^n=\bL^n$, the set $\Sigma$ is not an algebraic subset of~$\R^{N}$, the union $\Sigma\cup(-\Sigma)$ is the algebraic subset of~$\R^{N}$ given by $m+r$ equations~\eqref{eq_system_L1} and~\eqref{eq_system_L2}. Hence, $\Sigma$ is the union of several connected components of an algebraic subset of~$\R^N$.

Let $\Omega=\Omega_{K,\X^n}\subseteq\R^N$ be the subsets consisting of all polyhedra with non-degenerate faces. Notice that the property of having non-degenerate faces depends only on the set of edge lengths~$\bell$. Namely, the set of positive numbers $\bell=(\ell_e)_{e\in\fE(K)}$ will be called \textit{non-degenerate} if for each simplex $\sigma=[v_0\ldots v_k]$ in~$K$, there exists a non-degenerate simplex~$[p_0\ldots p_k]$ in~$\X^n$ with edge lengths $\dist_{\X^n}(p_i,p_j)=\ell_{[v_iv_j]}$. If $\bell$ is non-degenerate, then all polyhedra in~$\Sigma_{K,\bell,\X^n}$ have non-degenerate faces, i.\,e., $\Sigma_{K,\bell,\X^n}\subseteq \Omega_{K,\X^n}$. Also, the set~$\Omega_{K,\X^n}$ is the union of the configuration spaces $\Sigma_{K,\bell,\X^n}$ over all non-degenerate~$\bell$.

Obviously, the group $\Isom(\X^n)$ of isometries of~$\X^n$ acts naturally on the space~$\Omega$, and every configuration space $\Sigma_{K,\bell,\X^n}$ is invariant under this action. A flexible polyhedron~$P_t$ is just a continuous path in $\Sigma_{K,\bell,\X^n}$ that does not stay in a single $\Isom(\X^n)$-orbit. Hence, the existence of a flexible polyhedron  $P\colon K\to\X^n$ with the set of edge lengths~$\bell$ is equivalent to the existence of a connected component of $\Sigma_{K,\bell,\X^n}$ consisting of several (hence, infinitely many) $\Isom(\X^n)$-orbits.

\subsection{Dihedral angles and Dehn invariant}
We denote by~$\fF_k(K)$ the set of $k$-dimen\-sional simplices of~$K$.

If a polyhedron $P\colon K\to\X^n$ is embedded, then, for each $\sigma\in\fF_{n-2}(K)$, we can consider the interior dihedral angle of the polyhedron~$P$ at the codimension~$2$ face~$P(\sigma)$. (If $\X^n=\bS^n$, then we should additionally agree which of the two components of $\bS^n\setminus P(K)$ is interior.) Unfortunately, there is no reasonable way to define the dihedral angles~$\varphi_{\sigma}(P)$ for all polyhedra $P\colon K\to\X^n$ with non-degenerate faces so that they be continuous real-valued functions on~$\Omega$. Nevertheless, we can define \textit{oriented dihedral angles} of~$P$ modulo~$2\pi$ that give continuous functions $\varphi_{\sigma}\colon\Omega\to\R/2\pi\Z$. The definition is as follows, cf.~\cite[Section~9]{Gai15b}. Let $\tau_1$ and $\tau_2$ be the two $(n-1)$-dimensional simplices of~$K$ containing~$\sigma$. Take any point~$\bx$ in~$ P(\sigma)$. Let $\bn_1$ and~$\bn_2$ be the unit vectors in the tangent spaces $T_{\bx} P(\tau_1)$ and~$T_{\bx} P(\tau_2)$ orthogonal to~$T_{\bx} P(\sigma)$ and pointing inside the simplices~$ P(\tau_1)$ and~$ P(\tau_2)$, respectively. Let $\bm_1$ and $\bm_2$ be the \textit{outer normal vectors\/} to $ P(\tau_1)$ and~$ P(\tau_2)$ at~$\bx$, i.\,e., the unit vectors in~$T_{\bx}\X^n$ orthogonal to $T_{\bx} P(\tau_1)$ and to~$T_{\bx} P(\tau_2)$, respectively, such that the product of the direction of~$\bm_i$ and the positive orientation of~$ P(\tau_i)$ yields the positive orientation of~$\X^n$. We say that the positive direction of rotation around~$ P(\sigma)$ is from $\bm_1$ to~$\bn_1$, and denote by~$\varphi_{\sigma}=\varphi_{\sigma}( P)$ the angle from~$\bn_1$ to~$\bn_2$ in this positive direction. This angle is a well-defined element of~$\R/2\pi\Z$. It is easy to see that~$\varphi_{\sigma}$ is independent of the choice of the point~$\bx$ and does not change if we interchange the simplices~$\tau_1$ and~$\tau_2$.

The \textit{Dehn invariant} of a polyhedron $P$ with non-degenerate faces is, by definition, an element of the group $\R\otimes(\R/\pi\Z)$ given by
\begin{equation}\label{eq_Dehn}
\D(P)=\sum_{\sigma\in\fF_{n-2}(K)}\vol_{n-2}(P(\sigma))\otimes\varphi_\sigma(P).
\end{equation} 
(Here and further, $\otimes$ always means tensor product over~$\Z$.)

If $P\colon K\to\X^n$ is an embedding and the orientations of~$K$ and~$\X^n$ agree in the sense that the vectors~$\bm_i$ defined above do indeed point outwards, then the interior dihedral angles~$\hphi_{\sigma}(P)$ of the polytope enclosed by the polyhedral surface~$P(K)$ belong to~$(0,2\pi)$ and represent~$\varphi_{\sigma}(P)$ in~$\R/2\pi\Z$. Then~\eqref{eq_Dehn} turns into the standard expression~\eqref{eq_Dehn_initial} for the Dehn invariant of an embedded polyhedron. 

\subsection{Scissors congruence and volume}\label{subsection_vol}
\label{subsection_scissors}
To introduce rigorously the concept of scissors congruence, it is convenient to work with the following definition of a polytope: A \textit{polytope} in~$\X^n$ is a compact subset $P\subseteq\X^n$ that can be written as the union of a finite number of non-degenerate $n$-dimensional convex simplices in~$\X^n$. If $P$, $P_1,\ldots,P_k$ are polytopes, $P=P_1\cup\cdots\cup P_k$, and the interiors of $P_1,\ldots,P_k$  are pairwise disjoint, we say that $P$ decomposes into $P_1,\ldots,P_k$ and write $P=P_1\sqcup\cdots\sqcup P_k$.
Two polytopes $P$ and~$Q$ are said to be \textit{scissors congruent} if they admit decompositions $P=P_1\sqcup\cdots\sqcup P_k$ and $Q=Q_1\sqcup\cdots\sqcup Q_k$ such that $P_i=g_i(Q_i)$ for some isometries $g_i\in\Isom(\X^n)$, $i=1,\ldots,k$.

The \textit{scissors congruence group}~$\CP(\X^n)$ is the quotient of the free Abelian group on symbols~$[P]$ for all polytopes~$P$ in~$\X^n$ modulo the relations:
\begin{align*}
[P]&=[P_1]+[P_2]&&\text{for }P=P_1\sqcup P_2,\\
[P]&=[g(P)]&&\text{for }g\in\Isom(\X^n).
\end{align*}

It is easy to see that $[P]=[Q]$ if and only if the polytopes~$P$ and~$Q$ are \textit{stably scissors congruent}, i.\,e., there exist polytopes~$P'$ and~$Q'$ such that $P'$ is scissors congruent to~$Q'$ and $P\sqcup P'$ is scissors congruent to~$Q\sqcup Q'$. By a theorem of Zylev~\cite{Zyl65} (see e.\,g.~\cite{Sah79} for a proof), stable scissors congruence implies scissors congruence. Hence $[P]=[Q]$ if and only if the polytopes~$P$ and~$Q$ are scissors congruent. 

It is a standard fact that there are homomorphisms
\begin{align*}
\CV&\colon\CP(\X^n)\to \R,&\D&\colon\CP(\X^n)\to \R\otimes(\R/\pi\Z)
\end{align*}
called \textit{volume} and \textit{Dehn invariant}, respectively, such that, for each polytope~$P$, $\CV([P])$ and~$\D([P])$ are the volume and the Dehn invariant of~$P$, respectively. 

The results of Sydler~\cite{Syd65} and Jessen~\cite{Jes68} imply that, for $n=3$ and~$n=4$, an element $\xi\in\CP(\E^n)$ is trivial if and only if $\CV(\xi)=0$ and~$\Delta(\xi)=0$. 
It is not known whether the same is true for $\E^n$, where $n\ge 5$, and for non-Euclidean spaces of constant curvature. 

We would like to extend the relation of scissors congruence to not necessarily embedded polyhedra $P\colon K\to\X^n$. More precisely, we would like to assign an element~$[P]\in\CP(\X^n)$ to any such polyhedron. This can be done using the following standard approach, which is usually used to define the generalized oriented volume of a not necessarily embedded polyhedron.

First, suppose that $\X^n$ is either~$\R^n$ or~$\Lambda^n$. Then the one-point compactification~$\X^n\cup\{\infty\}$ of~$\X^n$ is homeomorphic to $n$-dimensional sphere.  For any point~$\bx\in \X^n\setminus P(|K|)$, we denote by~$\lambda(\bx)$ the \textit{linking number\/} of the pair of points~$\{\bx,\infty\}$ and the $(n-1)$-dimensional singular cycle~$P(K)$. By definition (cf.~\cite[Sect.~77]{SeTh80}), this linking number is equal to the algebraic intersection number of a generic piecewise smooth curve~$\gamma$ going from~$\bx$ to~$\infty$ and the cycle $P(K)$.  It is a standard fact that this algebraic intersection number is independent of the choice of~$\gamma$. Then $\lambda(\bx)$ is 
an almost everywhere defined piecewise constant function on~$\X^n$ with compact support.
For every integer~$k$, we denote by $R_k$ the closure of the set of all points $\bx\in\X^n$ such that $\lambda(\bx)=k$. It is easy to check that $R_k$ is a polytope unless $k=0$. We put,
$$
[P]=\sum_{k\ne 0}k[R_k] \in\CP(\X^n).
$$
The intuition behind this definition is that the polyhedral surface $P(K)$ encloses every polytope~$R_k$ exactly $k$ times.

Two polyhedra $P_1\colon K_1\to\X^n$ and $P_2\colon K_2\to\X^n$ are said to be \textit{scissors congruent} if and only if $[P_1]=[P_2]$ in~$\CP(\X^n)$. If $P\colon K\to\X^n$ is an embedded polyhedron, then, up to sign depending on the orientation, $[P]$ coincides with the element of~$\CP(\X^n)$ corresponding to the polytope enclosed by the embedded polyhedral surface~$P(K)$. So for embedded polyhedra, the above definition of  scissors congruence is equivalent to the usual definition of scissors congruence for polytopes.

For any polyhedron $P\colon K\to\X^n$, the value of the homomorphism $\CV\colon\CP(\X^n)\to \R$ on the element~$[P]$ is equal to
$$
\CV(P)=\sum_{k\in\Z}k\vol_n(R_k)=\int_{\X^n}\lambda(\bx)\,d\vol_{\X^n},
$$
which coincides with the standard definition of the \textit{generalized oriented volume} of a not necessarily embedded polyhedron. (Here we denote by $d\vol_{\X^n}$ the standard volume element in~$\X^n$.) Also, it is easy to check that the value of the homomorphism $\D\colon\CP(\X^n)\to \R\otimes(\R/\pi\Z)$ on the element~$[P]$ is equal to
the Dehn invariant~$\D(P)$ given by~\eqref{eq_Dehn}.

Second, suppose that $\X^n=\bS^n$. If a polyhedron $P\colon K\to\bS^n$ is contained in  open hemisphere~$\bS^n_+$, then we can consider the one-point compactification of~$\bS^n_+$, and repeat the above construction literally.  

Assume that the polyhedron~$P$ is not contained in~$\bS^n_+$. Then there is no canonical way to assign an element~$[P]\in\CP(\bS^n)$ to~$P$. Indeed, even for an embedded polyhedron, one has two possibilities to choose which of the two components of $\bS^n\setminus P(K)$ should be taken for the polytope enclosed by~$P(K)$. Nevertheless, there is a canonical way to assign to~$P$ an element $[P]\in \CP'(\bS^n)$, where $\CP'(\bS^n)$ is the quotient of~$\CP(\bS^n)$ by the infinite cyclic group generated by~$[\bS^n]$. (Notice that, according to the above definition, the whole sphere~$\bS^n$ is a polytope, hence, an element~$[\bS^n]\in\CP(\bS^n)$ is well defined.) Namely, we can take for~$\infty$ an arbitrary point in~$\bS^n\setminus P(K)$, repeat the above construction, and obtain an element $[P]\in\CP(\bS^n)$. Changing the point~$\infty$, we arrive to other elements of~$\CP(\bS^n)$. However, it is not hard to check that the difference of any two such elements of~$\CP(\bS^n)$ is a multiple of~$\bS^n$. Therefore, the image of~$[P]$ in~$\CP'(\bS^n)$ is well defined. 

The volume homomorphism $\CV\colon\CP(\bS^n)\to\R$ induces a well-defined homomorphism $\CV\colon\CP(\bS^n)\to\R/\sigma_n\Z$, where 
$$
\sigma_n=\vol_n(\bS^n)=\frac{2\pi^{\frac{n+1}2}}{\Gamma\left(\frac{n+1}2\right)}\ .
$$
Thus the generalized oriented volume~$\CV(P)$ of a polyhedron $P\colon K\to\bS^n$  is a well-defined element of~$\R/\sigma_n\Z$.

Since $\D([\bS^n])=0$, the Dehn invariant $\D([P])$ is well defined and again coincides with the Dehn invariant~$\D(P)$ given by~\eqref{eq_Dehn}. 

\subsection{Schl\"afli's formula} Schl\"afli's formula is a fundamental relation between the differentials of dihedral angles of a polyhedron that is deformed continuously preserving the combinatorial type. For a polyhedron $P\colon K\to\X^n$, this formula has the form:
\begin{equation}\label{eq_schlaefli}
\sum_{\sigma\in\fF_{n-2}(K)}\vol_{n-2}(P(\sigma))\,d\varphi_{\sigma}(P)=\left\{
\begin{aligned}
&0&&\text{if }\X^n=\E^n,\\
&\varepsilon (n-1)\, d\CV(P)&&\text{if }\X^n=\bS^n\text{ or }\bL^n. 
\end{aligned}
\right.
\end{equation} 

\subsection{Stratifications of real affine algebraic sets and holomorphic functions} 
Suppose that $\K$ is either~$\R$ or~$\C$, and let $X\subset\K^N$ be an irreducible affine variety of dimension~$d$. Let~$I_{X}\subseteq \K[x_1,\ldots,x_N]$ be the ideal   consisting of all polynomials that vanish identically on~$X$, and let $\K[X]=\K[x_1,\ldots,x_N]/I_{X}$ be  the ring of regular functions on~$X$. 

The variety~$X$ has two natural topologies, the \textit{analytic topology} induced by the standard Hausdorff topology in~$\K^N$ and the \textit{Zariski topology}. We agree that, unless otherwise stated explicitly, all topological properties (e.\,g. `open') should be understood with respect to the analytic topology. We shall always write explicitly `Zariski open' if we mean open with respect to the Zariski topology.

Recall that a point $x\in X$ is said to be \textit{regular} if the intersection of the kernels of the differentials $df|_x$ for all $f\in I_X$ has dimension~$d$, and is said to be \textit{singular} otherwise. We denote by~$X^{\reg}$ and~$X^{\sing}$ the sets of regular and singular points of~$X$, respectively. The set $X^{\reg}$ has a natural structure of a $d$-dimensional $\K$-analytic manifold. Hence, we can consider $\K$-analytic functions on open subsets of~$X^{\reg}$. If $\K=\C$ and $U$ is an open subset of~$X^{\reg}$, then we denote by $\CO(U)$ the ring of holomorphic (i.\,e. single-valued complex analytic) functions on~$U$. We shall need the following consequence of Liouville's theorem on entire functions, see~\cite[Lemma~9.7]{Gai15b}.

\begin{lem}\label{lem_log}
Suppose that $X$ is a smooth irreducible complex affine variety, $\psi\in\CO(X)$, and there exist non-zero regular functions $g_1,\ldots,g_N\in\C[X]$ such that 
\begin{equation*}
\Im \psi(x)\le \max_{n=1,\ldots,N} \ln|g_n(x)|
\end{equation*}
for all $x\in X$.  Then $\psi$ is a constant. \textnormal{(}Here we use the convention $\ln 0=-\infty$.\textnormal{)}
\end{lem}

Now, suppose that $X\subseteq\K^N$ is an arbitrary algebraic subset and $X_1,\ldots,X_k$ are the irreducible components of~$X$. Recall that a point $x\in X$ is said to be \textit{regular} if it is a regular point of certain~$X_i$ and does not belong to all other~$X_j$'s.  By definition, $\dim X=\max_{i=1,\ldots,k}\dim X_i$.

 Any real algebraic subset $X\subseteq\R^N$ possesses a natural stratification that can be constructed in the following way, see~\cite[Section 11(b)]{Whi57}. Let $d=\dim X$. Then $X=Y\cup Z$, where $Y$ is the the union of all $d$-dimensional irreducible components of~$X$, and $Z$ is the union of all irreducible components of~$X$ of dimensions strictly less than~$d$. Put $M=Y^{\reg}\setminus Z$ and $X'=Y^{\sing}\cup Z$; then $X=M\cup X'$ and $M\cap X'=\varnothing$. Then $M$ is a smooth $d$-dimensional manifold without boundary. It is proved in~\cite{Whi57} that $M$ consists of finitely many connected components.  All connected components of~$M$ are taken for $d$-dimensional strata of the stratification.
 Then the same procedure is recursively applied to~$X'$, which is a real algebraic subset of~$\R^N$ of dimension strictly less than~$d$. The obtained stratification of~$X$ consists of finitely many strata. It will be called the \textit{standard stratification\/} of~$X$. By construction, every stratum~$S$ of this stratification is a connected open subset of the set of regular points of an irreducible real affine variety. This easily implies the following lemma, cf.~\cite[Proposition~8.2]{Gai15b}.

\begin{lem}\label{lem_strat}
Let $S$ be a stratum of the standard stratification of an algebraic subset~$X\subseteq\R^N,$ and let $\Xi$ be the Zariski closure of~$S$ in~$\C^N$. Then $\Xi$ is an irreducible complex affine variety, $\dim_{\C}\Xi=\dim_{\R}S,$ and $S\subseteq \Xi^{\reg}$. 
\end{lem}

Suppose that $W$ is an open subset of~$\R^k$ and $\psi$ is a real analytic function defined on~$W$. Consider a path $\gamma$ in~$\C^k\supset\R^k$ that starts at a point $a\in W$, and try to continue analytically the function~$\psi$ along~$\gamma$. We can always perform this continuation at least until the path~$\gamma$ stays inside the disc of convergence of the Taylor series for~$\psi$ with centre~$a$. Later a singularity can occur. However, the following property will be crucial for us: \textit{The analytic continuation of~$\psi$ along~$\gamma$ is unique whenever exists.} Lemma~\ref{lem_strat} easily implies 
that every point $x\in S$ has a neighborhood~$U$ in~$\Xi^{\reg}$ such that the pair $(U,U\cap S)$ is biholomorphic to a pair $(V,V\cap\R^k)$ for some open subset $V\subseteq\C^k$, where $k=\dim_{\R}S$. Hence, the same property holds for the analytic continuation from~$S$ to~$\Xi^{\reg}$.

\begin{lem}\label{lem_strat_contin}
Let $S$ and~$\Xi$ be as in Lemma~\ref{lem_strat}. Suppose that $W$ is an open subset of~$S$, $\psi$ is a real analytic function defined on~$W$, and $\gamma$ is a path that starts in~$W$ and is contained in~$\Xi^{\reg}$. Then the analytic continuation of~$\psi$ along~$\gamma$ is unique whenever exists.
\end{lem}

\begin{remark}
The crucial point here is that $\dim_{\C}\Xi=\dim_{\R}S$. If the complex dimension of a complex variety~$\Xi$ were greater than the real dimension of a real subvariety $S\subset\Xi$, then it would be possible that a real analytic function on~$S$ has infinitely many different continuations along a path in~$\Xi$. For instance, one can consider continuations of a real analytic function from~$\R$ to~$\C^2\supset\C\supset\R$.
\end{remark}

If $X_0$ is a connected component of a real algebraic set $X\subseteq\R^N$, then the restriction of the standard stratification of~$X$ to~$X_0$ will be called the \textit{standard stratification} of~$X_0$.

\section{Proof of the main result}

The generalized oriented volume~$\CV$, the Dehn invariant~$\Delta$, and the oriented dihedral angles~$\varphi_{\sigma}$ are functions on the space~$\Omega_{K,\X^n}\subseteq\R^N$ of all polyhedra $P\colon K\to\X^n$ with non-degenerate faces. The following theorem is a more precise reformulation of Theorems~\ref{thm_main}
and~\ref{thm_conditional}.

\begin{theorem}\label{thm_precise}
Let $\X^n$ be one of the spaces~$\E^n$, $\bS^n$, and~$\bL^n$, where $n\ge 3$. Let $K$ be an oriented $(n-1)$-dimensional pseudo-manifold, let $\bell$ be a non-degenerate set of edge lengths for~$K$, and let $\Sigma_0$ be a connected component of~$\Sigma=\Sigma_{K,\bell,\X^n}$. Assume that either $\X^n=\E^n$ or the generalized oriented volume $\CV$ is constant on~$\Sigma_0$. Then the Dehn invariant $\Delta$ is also constant on~$\Sigma_0$.
\end{theorem}

Below we always assume that the $5$-tuple $(\X^n,K,\bell,\Sigma,\Sigma_0)$ is fixed and satisfies the hypothesis of Theorem~\ref{thm_precise}. 

For each $\sigma\in\fF_{n-2}$, the volume $\vol_{n-2}(P(\sigma))$ depends only on the edge lengths of the polyhedron~$P$, hence, is constant on~$\Sigma$. We denote this volume by~$V_{\sigma}$.

We need the following lemma, which is proved in~\cite[Lemma~9.2]{Gai15b} in the case~$\X^n=\bL^n$. The proofs in the cases $\X^n=\E^n$ and~$\X^n=\bS^n$ are quite similar, so we omit them.

\begin{lem}\label{lem_polyn}
For each $\sigma\in\fF_{n-2}(K)$, there exists a polynomial~$Q_{\sigma}$ in variables~$x_{v,j}$ such that the restriction of the function $\exp(i\varphi_{\sigma})$ to~$\Sigma$ coincides with the restriction of~$Q_{\sigma}$ to~$\Sigma$.
\end{lem} 

\begin{remark}
If $\X^n$ is either $\E^n$ or~$\bS^n$, this lemma claims that the restriction of the function $\exp(i\varphi_{\sigma})$ to $\Sigma$ is a complex-valued regular function on~$\Sigma$, i.\,e., belongs to~$\C[\Sigma]$. However, if $\X^n=\bL^n$, then we cannot formulate Lemma~\ref{lem_polyn} in this way, since $\Sigma$ is not an algebraic subset of~$\R^N$ but only the union of several connected components of an algebraic subset.
\end{remark}

The polynomials~$Q_{\sigma}$ are by no means unique. We choose some polynomials satisfying the required conditions, and fix this choice.

\begin{propos}\label{propos_main}
Let $U$ be a connected open subset of~$\Sigma_0$ such that every function $\varphi_{\sigma}\colon\Sigma_0\to\R/2\pi\Z$ has a continuous $\R$-valued branch~$\hphi_{\sigma}$ on~$U$. Let $f\colon\R\to\Q$ be an arbitrary $\Q$-linear mapping. Then the $\R$-valued function
\begin{equation}\label{eq_psif}
\psi_f=\sum_{\sigma\in \mathbf{F}_{n-2}(K)}f(V_{\sigma})\,\hphi_{\sigma}
\end{equation}
is constant on~$U$.
\end{propos}

\begin{proof}
Let $S$ be a stratum of the standard stratification of~$\Sigma_0$, and let $W$ be a connected component of~$U\cap S$. 

First, let us prove that the function~$\psi_f$ is constant on~$W$.

Let $\Xi$ be the Zariski closure of~$S$ in~$\C^N$. By Lemma~\ref{lem_strat}, $\Xi$ is an irreducible complex affine variety and $S\subseteq\Xi^{\reg}$. Let $\Upsilon\subseteq\Xi^{\reg}$ be the Zariski open subset consisting of all points $P$ such that $Q_{\sigma}(P)\ne 0$ for all~$\sigma\in\fF_{n-2}(K)$. The set~$\Upsilon$ contains~$W$, hence, is non-empty. Thus $\Upsilon$ is a connected complex analytic manifold. 

Let $\gamma$ be an oriented loop in~$\Upsilon$. Then $Q_{\sigma}(\gamma)$ is an oriented loop in~$\C$ that does not pass through~$0$. We denote by~$k_{\sigma}(\gamma)$  the winding number of the loop~$Q_{\sigma}(\gamma)$ around~$0$.

Each function~$\hphi_{\sigma}$ is a branch of the multi-valued analytic function $-i\Ln Q_{\sigma}$ on the set~$W$. Since $Q_{\sigma}$ is regular and does not take zero value in~$\Upsilon$, we see that the function~$\hphi_{\sigma}$ admits an analytic continuation~$\hphi_{\sigma}^{(\gamma)}$ along any path~$\gamma$ in~$\Upsilon$ starting in~$W$, and this analytic continuation is a branch of the multi-valued analytic function $-i\Ln Q_{\sigma}$.  If $\gamma$ is an oriented loop in~$\Upsilon$ with base point in~$W$, then continuing analytically the function~$\hphi_{\sigma}$ along~$\gamma$, we arrive to another branch of the multi-valued analytic function $-i\Ln Q_{\sigma}$ on~$W$, namely, to the branch $\hphi_{\sigma}+2\pi k_{\sigma}(\gamma)$.

It follows from Schl\"afli's formula~\eqref{eq_schlaefli} and the hypothesis of Theorem~\ref{thm_precise} that 
$$
\left.\sum_{\sigma\in \fF_{n-2}(K)}V_{\sigma}\,d\hphi_{\sigma}\right|_{TW}=0,
$$
where $TW$ is the tangent bundle of~$W$.
Therefore, the function
$
\sum_{\sigma\in \fF_{n-2}(K)}V_{\sigma}\,\hphi_{\sigma}
$
is constant on~$W$. By Lemma~\ref{lem_strat_contin}, the analytic continuation of this function along a loop~$\gamma$ in~$\Upsilon$ is unique. Hence, the function $
\sum_{\sigma\in \fF_{n-2}(K)}V_{\sigma}\,\hphi_{\sigma}^{(\gamma)}$ is constant along~$\gamma$. Therefore,
\begin{equation*}
\sum_{\sigma\in \fF_{n-2}(K)}V_{\sigma}k_{\sigma}(\gamma)=0
\end{equation*}
for all loops~$\gamma$ in~$\Upsilon$. Since $k_{\sigma}(\gamma)\in\Z$ and the mapping~$f$ is $\Q$-linear, we obtain that
\begin{equation*}
\sum_{\sigma\in \fF_{n-2}(K)}f(V_{\sigma})k_{\sigma}(\gamma)=0.
\end{equation*}
It follows that  continuing analytically the function~$\psi_f$ along every loop in~$\Upsilon$ with base point in~$W$, we arrive to the function~$\psi_f$ again. In other words, the function~$\psi_f$ can be continued to a single-valued holomorphic function on~$\Upsilon$, which we again denote by~$\psi_f$.

Since $\psi_f$ is a branch of the multi-valued function $$-i\sum_{\sigma\in \fF_{n-2}(K)}f(V_{\sigma})\Ln Q_{\sigma}\,,$$ for all $P\in\Upsilon$
we have the following estimate 
\begin{multline*}
|\Im\psi_f(P)|\le\sum_{\sigma\in \fF_{n-2}(K)}|f(V_{\sigma})|\,\bigl|\ln|Q_{\sigma}( P)|\bigr|
\le \max_{\sigma\in \fF_{n-2}(K)}\left(M|f(V_{\sigma})|\,\bigl|\ln|Q_{\sigma}( P)|\bigr|\right)\\
{}\le\max\left(\max_{\sigma\in \fF_{n-2}(K)}\ln\left|Q_{\sigma}( P)^L\right|,\max_{\sigma\in \fF_{n-2}(K)}\ln\left|Q_{\sigma}( P)^{-L}\right|\right),
\end{multline*}
where $M$ is the number of $(n-2)$-dimensional simplices of~$K$, and $L$ is a positive integer that is greater than all numbers $M|f(V_{\sigma})|$, where $\sigma\in\fF_{n-2}(K)$.

Take a principal Zariski open subset $\Xi_h\subseteq\Xi$ such that $\Xi_h\subseteq\Upsilon$. Then  $\Xi_h$ is a smooth irreducible affine variety, and both $Q_{\sigma}$ and $Q_{\sigma}^{-1}$ are regular functions on~$\Xi_h$. Applying Lemma~\ref{lem_log} to the restriction of~$\psi_f$ to~$\Xi_h$, we obtain that~$\psi_f$ is  constant on~$\Xi_h$, hence, it is constant on~$\Upsilon$. Thus, $\psi_f$ is constant on~$W$.

Second, let us prove that the function~$\psi_f$ is constant on the whole set~$U$. The standard stratification of~$\Sigma$ consists of finitely many strata~$S$, and every intersection~$U\cap S$ is open in~$S$, hence, consists of at most countably many connected components. Since $\psi_f$ is constant on every connected component of every~$U\cap S$, it follows that the function~$\psi_f$ takes at most countably many values. Since $U$ is connected and $\psi_f$ is continuous, we obtain that $\psi_f$ is constant on~$U$.
\end{proof}

\begin{proof}[Proof of Theorem~\ref{thm_precise}]
Since the set~$\Sigma_0$ is connected, we suffice to prove that every point $P_0\in\Sigma_0$ has a neighborhood~$U$ in~$\Sigma_0$ such that the function $\Delta(P)$ is constant on~$U$. Obviously, every point $P_0\in\Sigma_0$ has a neighborhood~$U$ in~$\Sigma_0$ such that every function $\varphi_{\sigma}\colon\Sigma_0\to\R/2\pi\Z$ has a continuous $\R$-valued branch~$\hphi_{\sigma}$ on~$U$. By Proposition~\ref{propos_main}, for each $\Q$-linear mapping~$f\colon\R\to\Q$, the function~$\psi_f$ given by~\eqref{eq_psif} is constant on~$U$. This implies that the mapping $\hD\colon U\to\R\otimes\R$ given by $P\mapsto V_{\sigma}\otimes\hphi_{\sigma}(P)$ is constant. Hence the composite mapping 
$$
\Delta\colon U\xrightarrow{\ \hD\ }\R\otimes\R\xrightarrow{\mathrm{id}\otimes\mathrm{pr}}\R\otimes(\R/\pi\Z),
$$
where $\mathrm{pr}\colon\R\to\R/\pi\Z$ is a projection, is also constant on~$U$.
\end{proof}

\section{On a counterexample by Alexandrov and Connelly}\label{section_example}

A \textit{suspension with a hexagonal equator} is a polyhedron in~$\R^3$ of the combinatorial type of a hexagonal bipyramid, i.\,e., a polyhedron with $8$ vertices $N$ (north pole), $S$ (south pole), $p_1,\ldots,p_6$ and $16$ faces $Np_jp_{j+1}$ and $Sp_jp_{j+1}$, $j=1,\ldots,6$. Here and further the sums of indices are always taken modulo~$6$, i.\,e., $j+1$ should be replaced by~$1$ whenever $j=6$. 

The following construction of a flexible suspension with a hexagonal equator is a partial case of the description of all flexible suspensions obtained by Connelly~\cite{Con74}.
\pagebreak

The initial data for this construction are a real non-degenerate cubic~$E$ given by  $$y^2=x(x-b')(x-b),$$ where $0<b'<b$, and $24$ points on it that are denoted by $Q_{j,\pm}$ and $Q_{j,\pm}'$, $j=1,\ldots,6$, and should satisfy the following conditions:
\begin{enumerate}
\item[(A)] The points~$Q_{j,+}$ and~$Q_{j,-}$ either coincide or are symmetric to each other about the $x$-axis, and the points~$Q_{j,+}'$ and~$Q_{j,-}'$ either coincide or are symmetric  to each other about the $x$-axis. 
\item[(B)] $Q_{j,-}+ Q_{j,-}' + Q_{j+1,+} + Q'_{j+1,+}=0$, where $+$ denotes the addition of points on the elliptic curve~$E$ and the infinite point of~$E$ is taken for the neutral element~$0$. Besides, no pair of the four points $Q_{j,-}$, $Q_{j,-}'$,  $Q_{j+1,+}$, and~$Q'_{j+1,+}$ is symmetric about the $x$-axis.
\item[(C)] The total collection of the $24$ points $Q_{j,\pm}$ and~$Q'_{j,\pm}$ (counting multiplicities) is symmetric about the $x$-axis.
\item[(D)] The points $Q'_{j,\pm}$ are on the compact component of~$E$ and the points~$Q_{j,\pm}$ are on the non-compact component of~$E$.
\end{enumerate}

It is well-known that condition~(B) is equivalent to the following: The four points $Q_{j,-}$, $Q_{j,-}'$, $Q_{j+1,+}$, and~$Q'_{j+1,+}$ lie on a parabola $y=q_j(x)$, where $q_j(x)=a_jx^2+b_jx+c_j$ is a quadratic polynomial with real coefficients and $a_j\ne 0$.
Denote by~$r_j$ the $x$-coordinate of the points~$Q_{j,\pm}$ and by~$r_j'$ the $x$-coordinate of the points~$Q_{j,\pm}'$; then $r_j\ge b$ and~$0\le r_j'\le b'$.
Since $r_j$, $r_j'$, $r_{j+1}$, and~$r_{j+1}'$ are the four roots of the polynomial $q_j(x)^2-x(x-b')(x-b)$, Vieta's formula implies that 
\begin{equation}\label{eq_Vieta}
\frac{c_j}{a_j}=s_j\sqrt{r_jr_j'r_{j+1}r_{j+1}'}\,,
\end{equation}
where $s_j=\pm 1$.

Consider the complex-valued functions defined on the interval $(b',b)$ given by
\begin{equation}\label{eq_Fj}
F_j(x)=\frac{-q_j(x)+i\sqrt{-x(x-b)(x-b')}}
{a_j\sqrt{(x-r_j)(x-r_j')(x-r_{j+1})(x-r_{j+1}')}}
\end{equation}
where the positive values of the square roots are taken.

Since the points $Q_{j,-}$, $Q_{j,-}'$, $Q_{j+1,+}$, and~$Q'_{j+1,+}$ lie on the parabola $y=q_j(x)$, it follows that $|F_j(x)|=1$ for all $x\in (b',b)$. 
Connelly showed that condition~(C) implies that 
\begin{equation}\label{eq_circ}
\prod_{j=1}^6F_j(x)=1.
\end{equation}
In particular, it follows that $s_1s_2\cdots s_6=1$. Hence, we can choose signs~$\sigma_j=\pm 1$ so that $\sigma_j\sigma_{j+1}=s_j$ for all~$j$. The numbers~$\sigma_j$ are well defined up to simultaneous change of signs  of  all of them, which geometrically corresponds to the interchange of the poles~$N$ and~$S$.
 
To write explicit formulae for the flexion, we shall conveniently identify $\R^3$ with $\C\times \R$. Now, the flexion of a suspension with a hexagonal equator is given by
 \begin{gather*}
S(x)=(0,0),\qquad N(x)=\bigl(0,\sqrt{x}\bigr),\\
p_k(x)=\frac{1}{2\sqrt{x}}\left(
\sqrt{-(x-r_k)(x-r_k')}\cdot\prod_{j=1}^{k-1}F_j(x),\ x+\sigma_k\sqrt{r_kr_k'}
\right),
\end{gather*}
where $x$ runs over the interval~$(b',b)$.

 Using these formulae and formulae~\eqref{eq_Vieta},~\eqref{eq_Fj}, and~\eqref{eq_circ}, it can be checked immediately that the lengths of the edges of the suspension are constant and are given by 
\begin{gather}\label{eq_edge1}
|p_k-S|=\frac{\sqrt{r_k}+\sigma_k\sqrt{r_k'}}{2}\qquad\qquad |p_k-N|=\frac{\sqrt{r_k}-\sigma_k\sqrt{r_k'}}{2}\\
\label{eq_edge2}|p_{k+1}-p_k|=\frac{1}{2|a_k|}
\end{gather}

\begin{remark}
The geometric interpretation of the numbers $F_j$ is as follows: $F_j=\exp(i\theta_j)$, where~$\theta_j$ is the oriented dihedral angle between the halfplanes $NSp_j$ and $NSp_{j+1}$. 
\end{remark}

Alexandrov and Connelly~\cite{AlCo11} considered a particular example of such flexible suspension with a hexagonal equator and claimed that its Dehn invariant is non-constant during the flexion. Namely, they fixed $b'=51$, $b=100$, considered the points 
\begin{align*}
A&=(2,98),& B&=\left(\frac{4039540}{762129}\,,\frac{100768585960}{665338617}\right),\\ C&=(102,-102),& D&=(30,-210)
\end{align*}
on the cubic $y^2=x(x-51)(x-100)$ and the $24$ points~$Q_{j,\pm}$, $Q'_{j,\pm}$ as shown in Table~\ref{table} (cf.~\cite[Table~2]{AlCo11}). Recall that given all these data, we still have two possibilities for the signs~$\sigma_j$. The suspension considered by Alexandrov and Connelly corresponds to the choice of $\sigma_1=-1$. All other signs can be restored uniquely: $\sigma_2= \sigma_3=\sigma_4=\sigma_5=1$, $\sigma_6=-1$.

\begin{table}
\caption{The points~$Q_{j,\pm}$ and~$Q'_{j,\pm}$}\label{table}

\begin{tabular}[t]{|c|c|c|c|c|}
\firsthline
$j$ & $\vphantom{\widehat{R}}Q_{j,-}$ & $Q_{j+1,+}$ & $Q'_{j,-}$ & $\vphantom{Q_{\mathstrut}'}Q'_{j+1,+}$ \\
\hline
1 & $\vphantom{\widehat{R}}A-B+C$ & $-A+2B-C-D$ & $-B$ & $D$ \\
2 & $\vphantom{\widehat{R}}A-2B+C+D$ & $2B-C-2D$ & $D$ & $-A$ \\
3 & $\vphantom{\widehat{R}}-2B+C+2D$ & $A+B-C-2D$ & $-A$ & $B$ \\
4 & $\vphantom{\widehat{R}}-A-B+C+2D$ & $A-C-D$ & $B$ & $-D$ \\
5 & $\vphantom{\widehat{R}}-A+C+D$ & $-C$ & $-D$ & $A$ \\
6 & $\vphantom{\widehat{R}}C$ & $-A+B-C$ & $A$ & $-B$ \\
\hline
\end{tabular}

\end{table}

Further, they computed the Dehn invariant of the obtained flexible suspension
\begin{multline*}
\D(x)=1\otimes\alpha_1(x)+\sqrt{2}\otimes\alpha_2(x)+\sqrt{15}\otimes\alpha_3(x)
+\sqrt{30}\otimes\alpha_4(x)\\{}+\sqrt{85}\otimes\alpha_5(x)+\sqrt{102}\otimes\alpha_6(x)
+\sqrt{170}\otimes\alpha_7(x),
\end{multline*}
where $\alpha_j(x)$, $j=1,\ldots,7$, are certain $\Q$-linear combinations of the oriented dihedral angles of the suspension, and concluded that the Dehn invariant~$\D(x)$ is constant in~$x$ if and only if $\alpha_j(x)$ is constant in~$x$ for every $j=1,\ldots,7$.

The linear combinations~$\alpha_j(x)$ were written explicitly in~\cite[(4.1)--(4.7)]{AlCo11}; in particular, 
$$
\alpha_4(x)=-\frac12\bigl(\varphi_2(x)-\varphi'_2(x)\bigr)-\frac12\bigl(\varphi_5(x)-\varphi'_5(x)\bigr),
$$
where $\varphi_j$ and $\varphi_j'$ are the oriented dihedral angles at the edges~$Np_j$ and~$Sp_j$, respectively. Besides, Alexandrov and Connelly showed that $\varphi_j(x)=-\varphi_j'(x)$ for all~$j$, hence,
$$
\alpha_4(x)=-\varphi_2(x)-\varphi_5(x).
$$
Further, they showed that the oriented dihedral angles $\varphi_{1,2}(x)$ and~$\varphi_{4,5}(x)$ at the edges~$p_1p_2$ and~$p_4p_5$ are equal to each other for all~$x$. 
For $x=51$, the suspension is flat, i.\,e., contained in a plane, and $\varphi_2(51)=\varphi_5(51)=\pi$ and $\varphi_{1,2}(51)=\varphi_{4,5}(51)=0$, see~\cite[Table~6]{AlCo11}. When $x$ increases from~$51$, the angles~$\varphi_{1,2}(x)=\varphi_{4,5}(x)$ increase. 

\begin{figure}
\begin{center}
\unitlength=3.5mm
\begin{picture}(20,9)
\put(0,0.4){\includegraphics[scale=1]{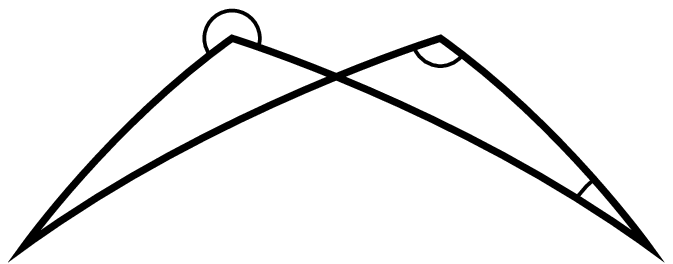}}
\put(6.3,8.4){$\varphi_2$}
\put(12.8,7.4){$\widetilde{S}$}
\put(6.5,5.3){$\widetilde{N}$}
\put(12.7,5.2){$\varphi_2'$}
\put(17.9,3.1){$\varphi_{1,2}$}
\put(-.4,0){$\widetilde{p}_3$}
\put(19.8,0){$\widetilde{p}_1$}
\end{picture}
\end{center}

\caption{Spherical quadrangle~$\mathcal{Q}_2$}\label{fig_p2}

\vspace{1.4cm}

\begin{center}
\unitlength=3.5mm
\begin{picture}(20,9)
\put(0,0.4){\includegraphics[scale=1]{fig1.eps}}
\put(6.3,8.4){$\varphi_5$}
\put(12.8,7.4){$\widetilde{S}$}
\put(6.5,5.3){$\widetilde{N}$}
\put(12.7,5.2){$\varphi_5'$}
\put(17.9,3.1){$\varphi_{4,5}$}
\put(-.4,0){$\widetilde{p}_6$}
\put(19.8,0){$\widetilde{p}_4$}
\end{picture}
\end{center}
\caption{Spherical quadrangle~$\mathcal{Q}_5$: Incorrect}\label{fig_p5}

\vspace{1.4cm}

\begin{center}
\unitlength=3.5mm
\begin{picture}(20,9)
\put(0,0.4){\includegraphics[scale=1]{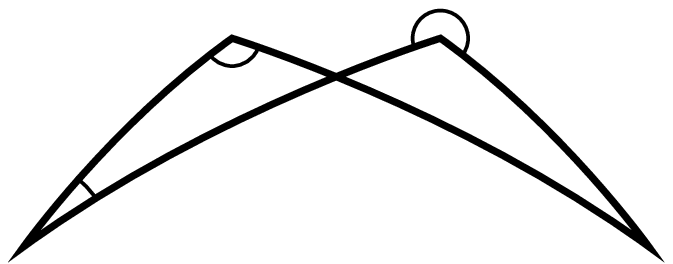}}
\put(6.5,5.3){$\varphi_5$}
\put(12.5,5.3){$\widetilde{S}$}
\put(6.6,7.4){$\widetilde{N}$}
\put(13,8.4){$\varphi_5'$}
\put(.5,3.1){$\varphi_{4,5}$}
\put(-.4,0){$\widetilde{p}_4$}
\put(19.8,0){$\widetilde{p}_6$}
\end{picture}
\end{center}
\caption{Spherical quadrangle~$\mathcal{Q}_5$: Correct}\label{fig_p5_corr}
\end{figure}

Alexandrov and Connelly claimed that both $\varphi_2(x)$ and~$\varphi_5(x)$ also increase as $x$ increases, hence, their sum cannot be constant. To show that $\varphi_2(x)$ and~$\varphi_5(x)$ increase, they studied spherical quadrangles~$\mathcal{Q}_2$ and~$\mathcal{Q}_5$ cut on the unit spheres with centres~$p_2(x)$ and~$p_5(x)$ by faces of the suspension. They showed that the opposite sides of either of  these quadrangles are equal to each other. Further, they argued that these two quadrangles look as it is shown in Figs.~\ref{fig_p2} and~\ref{fig_p5}, respectively. (These figures are exactly Figs.~14 and~15 in~\cite{AlCo11}. Here $\widetilde{N}$, $\widetilde{S}$, and~$\widetilde{p}_j$ indicate the vertices of the quadrangles corresponding to the edges leading to $N$, $S$, and~$p_j$, respectively.) If the spherical quadrangles~$\mathcal{Q}_2$ and~$\mathcal{Q}_5$ were as in Figs.~\ref{fig_p2} and~\ref{fig_p5} then the increase of the angles~$\varphi_{1,2}(x)=\varphi_{4,5}(x)$ would imply the increase of the angles~$\varphi_2(x)$ and~$\varphi_5(x)$. Nevertheless, the latter figure is incorrect. Indeed, the edge lengths of the suspension can be computed explicitly using~\eqref{eq_edge1},~\eqref{eq_edge2}; the result is given in~\cite[Tables~4 and~5]{AlCo11}. Now, one can easily compute: 
$$
\begin{aligned}
\cos{\angle Np_5 p_6}=\cos{\angle S p_5 p_4}&=\frac{437 \sqrt{170}-121 \sqrt{30}}{6500}\\
\cos{\angle S p_5 p_6}=\cos{\angle Np_5 p_4}&=\frac{437 \sqrt{170}+121 \sqrt{30}}{6500}\\
\end{aligned}
$$
Hence, in the spherical quadrangle~$\mathcal{Q}_5$, the lengths of the edges~$\widetilde{N}\widetilde{p}_4$ and~$\widetilde{S}\widetilde{p}_6$  are smaller than the lengths of the edges~$\widetilde{N}\widetilde{p}_6$ and~$\widetilde{S}\widetilde{p}_4$. Therefore, Fig.~\ref{fig_p5} is incorrect. In fact, the spherical quadrangle~$\mathcal{Q}_5$ is as shown in Fig.~\ref{fig_p5_corr}. Hence, the angle $\varphi_5(x)$ decreases as $x$ increases. So the argumentation in~\cite{AlCo11} fails.

\end{document}